\documentclass[oneside,english]{amsart}
\usepackage[T1]{fontenc}
\usepackage[latin9]{inputenc}
\usepackage{amsthm}
\usepackage{amssymb}

\makeatletter
\numberwithin{equation}{section}
\numberwithin{figure}{section}
\theoremstyle{plain}
\newtheorem{thm}{\protect\theoremname}
  \theoremstyle{plain}
  \newtheorem{lem}[thm]{\protect\lemmaname}
  \theoremstyle{plain}
  \newtheorem{conjecture}[thm]{\protect\conjecturename}
  \theoremstyle{remark}
  \newtheorem{rem}[thm]{\protect\remarkname}

\makeatother

  \providecommand{\conjecturename}{Conjecture}
  \providecommand{\lemmaname}{Lemma}
  \providecommand{\remarkname}{Remark}
\providecommand{\theoremname}{Theorem}

\begin{document}
\author{Minoru Hirose}
\address[Minoru Hirose]{Faculty of Mathematics, Kyushu University
 744, Motooka, Nishi-ku, Fukuoka, 819-0395, Japan}
\email{m-hirose@math.kyushu-u.ac.jp}
\author{Hideki Murahara}
\address[Hideki Murahara]{Nakamura Gakuen University Graduate School,
 5-7-1, Befu, Jonan-ku, Fukuoka, 814-0198, Japan} 
\email{hmurahara@nakamura-u.ac.jp}
\author{Shingo Saito}
\address[Shingo Saito]{Faculty of Arts and Science, Kyushu University,
 744, Motooka, Nishi-ku, Fukuoka, 819-0395, Japan}
\email{ssaito@artsci.kyushu-u.ac.jp}

\title{Weighted sum formula for multiple harmonic sums modulo primes}

\date{\today}
\begin{abstract}
In this paper we prove a weighted sum formula for multiple harmonic
sums modulo primes, thereby proving a weighted sum formula for finite
multiple zeta values. Our proof utilizes difference equations for
the generating series of multiple harmonic sums. We also conjecture
several weighted sum formulas of similar flavor for finite multiple
zeta values.
\end{abstract}
\maketitle

\section{Introduction}

For integers $k_{1},\dots,k_{d}\geq1$ and a prime $p$, we put
\begin{align*}
H_{p}(k_{1},\dots,k_{d}) & :=\sum_{0<n_{1}<\cdots<n_{d}<p}n_{1}^{-k_{1}}\cdots n_{d}^{-k_{d}}\in\mathbb{F}_{p}:=\mathbb{Z}/p\mathbb{Z},\\
H_{p}^{\star}(k_{1},\dots,k_{d}) & :=\sum_{0<n_{1}\leq\cdots\leq n_{d}<p}n_{1}^{-k_{1}}\cdots n_{d}^{-k_{d}}\in\mathbb{F}_{p}:=\mathbb{Z}/p\mathbb{Z}.
\end{align*}
The purpose of this paper is to prove the following theorem:
\begin{thm}[Weighted sum formula for multiple harmonic sums modulo primes]
\label{thm:main_harmonic_sum}Let $k$ be a positive integer, $d$
a positive odd integer, $p$ an odd prime larger than $d$, and $i$
an integer with $1\leq i\leq d$. Then we have
\[
\sum_{\substack{k_{1},\dots,k_{d}\geq1\\
k_{1}+\cdots+k_{d}=k
}
}2^{k_{i}}H_{p}(k_{1},\dots,k_{d})=\begin{cases}
0 & p-1\nmid k,\\
-1 & p-1\mid k.
\end{cases}
\]

\end{thm}
We put $\mathcal{A}:=\prod_{p}\mathbb{F}_{p}/\bigoplus_{p}\mathbb{F}_{p}$.
For $k_{1},\dots,k_{d}\in\mathbb{Z}_{\geq1}$, Kaneko and Zagier \cite{KZ}
defined $\mathcal{A}$-finite multiple zeta(-star) values by
\begin{align*}
\zeta_{\mathcal{A}}(k_{1},\dots,k_{d}) & :=(H_{p}(k_{1},\dots,k_{d}))_{p}\in\mathcal{A},\\
\zeta_{\mathcal{A}}^{\star}(k_{1},\dots,k_{d}) & :=(H_{p}^{\star}(k_{1},\dots,k_{d}))_{p}\in\mathcal{A}.
\end{align*}
These are analogues of classical real-valued multiple zeta(-star)
values. Theorem \ref{thm:main_harmonic_sum} implies the following
corollary:
\begin{thm}[Weighted sum formula for finite multiple zeta(-star) values]
\label{thm:main_finite_mzv}Let $k$ be a positive integer, $d$
a positive odd integer, and $i$ an integer with $1\leq i\leq d$.
Then we have
\end{thm}
\begin{align}
\sum_{\substack{k_{1},\dots,k_{d}\geq1\\
k_{1}+\cdots+k_{d}=k
}
}2^{k_{i}}\zeta_{\mathcal{A}}(k_{1},\dots,k_{d}) & =0,\label{eq:main_A}
\end{align}
\begin{align}
\sum_{\substack{k_{1},\dots,k_{d}\geq1\\
k_{1}+\cdots+k_{d}=k
}
}2^{k_{i}}\zeta_{\mathcal{A}}^{\star}(k_{1},\dots,k_{d}) & =0.\label{eq:main_A_star}
\end{align}

Note that several similar weighted sum formulas for multiple zeta
values are already known \cite{GuoLiXie,OhnoZudilin,OngEieLiaw},
and somewhat different weighted sum formulas for $\mathcal{A}$-finite
multiple zeta values were proved by Kamano in \cite{Kamano}.

In Section \ref{sec:proof}, we prove Theorems \ref{thm:main_harmonic_sum}
and \ref{thm:main_finite_mzv}. In Section \ref{sec:Further-conjecture},
we conjecture several weighted sum formulas of similar flavor.

\section{\label{sec:proof}proof of main theorem}

Throughout this section, we fix an arbitrary odd prime $p$.

\subsection{\label{sub:Restatement}Restatement of the main theorem by a generating
function}

For $x_{1},\dots,x_{d}\in\mathbb{F}_{p}(x)\setminus\mathbb{F}_{p}$,
we put
\[
f(x_{1},\dots,x_{d}):=\sum_{0<n_{1}<\cdots<n_{d}<p}\frac{1}{(n_{1}-x_{1})\cdots(n_{d}-x_{d})}\in\mathbb{F}_{p}(x).
\]
Here we understand that $f(x_{1},\dots,x_{d})=1$ if $d=0$. Then
since $1/(n-x)=\sum_{k=1}^{\infty}x^{k-1}n^{-k}$, the series expansion
of $f(a_{1}x,\dots,a_{d}x)$ at $x=0$ is given by
\[
\sum_{k_{1},\dots,k_{d}\geq1}a_{1}^{k_{1}-1}\cdots a_{d}^{k_{d}-1}H_{p}(k_{1},\dots,k_{d})x^{k_{1}+\cdots+k_{d}-d}\in\mathbb{F}_{p}[[x]]
\]
for $a_{1},\dots,a_{d}\in\mathbb{F}_{p}^{\times}$. Therefore, 
\[
\sum_{\substack{k_{1},\dots,k_{d}\geq1\\
k_{1}+\cdots+k_{d}=k
}
}2^{k_{i}}H_{p}(k_{1},\dots,k_{d})
\]
is equal to the coefficient of $x^{k}$ in
\[
2x^{d}f(\overbrace{x,\dots,x}^{i-1},2x,\overbrace{x,\dots,x}^{d-i})
\]
for every $k\geq0$. Thus Theorem \ref{thm:main_harmonic_sum} is
equivalent to the statement that
\[
f(\overbrace{x,\dots,x}^{i-1},2x,\overbrace{x,\dots,x}^{d-i})=\frac{x^{p-1-d}}{2(x^{p-1}-1)}
\]
holds for every positive odd integer $d<p$ and $i\in\{1,\dots,d\}$.

\subsection{Evaluation of $f(x,\dots,x)$}
\begin{lem}
\label{lem:Fxx}For $0<d<p$, we have
\[
f(\overbrace{x,\dots,x}^{d})=\frac{x^{p-1-d}}{x^{p-1}-1}.
\]
\end{lem}
\begin{proof}
Put $a_{0}=1$ and $a_{d}=f(\overbrace{x,\dots,x}^{d})$ for $0<d<p$.
Then we have
\[
\prod_{0<n<p}\left(1+\frac{y}{n-x}\right)=\sum_{d=0}^{p-1}a_{d}y^{d}.
\]
On the other hand,
\begin{align*}
\prod_{0<n<p}\left(1+\frac{y}{n-x}\right) & =\prod_{0<n<p}\frac{n-x+y}{n-x}\\
 & =\frac{(x-y)^{p-1}-1}{x^{p-1}-1}\\
 & =1+\sum_{d=1}^{p-1}\frac{x^{p-1-d}}{x^{p-1}-1}y^{d}
\end{align*}
in $\mathbb{F}_{p}(x)[y]$. Thus $a_{d}=x^{p-1-d}/(x^{p-1}-1)$ for
$0<d<p$.
\end{proof}

\subsection{Difference equations}

Put $\tilde{f}(x_{0},x_{1},\dots,x_{d},x_{d+1}):=\frac{1}{x_{0}x_{d+1}}f(x_{1},\dots,x_{d})$
for $x_{0},\dots,x_{d+1}\in\mathbb{F}_{p}(x)\setminus\mathbb{F}_{p}$.
For $a,b\in\mathbb{Z}_{\geq-1}$ with $(a,b)\neq(-1,-1)$, we put
\begin{align*}
G_{a,b}(x):= & x^{2}\tilde{f}(\overbrace{x,\dots,x}^{a+1},2x,\overbrace{x,\dots,x}^{b+1})\\
= & \begin{cases}
f(\overbrace{x,\dots,x}^{a},2x,\overbrace{x,\dots,x}^{b}) & a\geq0\ {\rm and}\ b\geq0,\\
\frac{1}{2}f(\overbrace{x,\dots,x}^{a+b+1}) & a=-1\ {\rm or}\ b=-1
\end{cases}
\end{align*}
and
\begin{align*}
S_{a,b}(x):= & x^{2}\tilde{f}(\overbrace{x,\dots,x}^{a+1},2x-1,\overbrace{x,\dots,x}^{b+1})\\
= & \begin{cases}
f(\overbrace{x,\dots,x}^{a},2x-1,\overbrace{x,\dots,x}^{b}) & a\geq0\ {\rm and}\ b\geq0,\\
\frac{x}{2x-1}f(\overbrace{x,\dots,x}^{a+b+1}) & a=-1\ {\rm or}\ b=-1.
\end{cases}
\end{align*}

\begin{lem}
\label{lem:all_translation}For $d\geq1$, we have
\begin{multline*}
f(x_{1}-1,\dots,x_{d}-1)-f(x_{1},\dots,x_{d})\\
=\frac{1}{x_{1}-1}f(x_{2}-1,\dots,x_{d}-1)-\frac{1}{x_{d}}f(x_{1},\dots x_{d-1}).
\end{multline*}
\end{lem}
\begin{proof}
Since
\begin{multline*}
\{1<n_{1}<\cdots<n_{d}\leq p\}\sqcup\{1=n_{1}<\cdots<n_{d}\leq p\}\\
=\{0<n_{1}<\cdots<n_{d}<p\}\sqcup\{0<n_{1}<\cdots<n_{d}=p\},
\end{multline*}
we have
\begin{align*}
 & f(x_{1}-1,\dots,x_{d}-1)\\
= & \sum_{1<n_{1}<\cdots<n_{d}\leq p}\frac{1}{(n_{1}-x_{1})\cdots(n_{d}-x_{d})}\\
= & \sum_{0<n_{1}<\cdots<n_{d}<p}\frac{1}{(n_{1}-x_{1})\cdots(n_{d}-x_{d})}+\sum_{0<n_{1}<\cdots<n_{d}=p}\frac{1}{(n_{1}-x_{1})\cdots(n_{d}-x_{d})}\\
 & -\sum_{1=n_{1}<\cdots<n_{d}\leq p}\frac{1}{(n_{1}-x_{1})\cdots(n_{d}-x_{d})}\\
= & f(x_{1},\dots,x_{d})-\frac{1}{x_{d}}f(x_{1},\dots x_{d-1})-\frac{1}{1-x_{1}}f(x_{2}-1,\dots,x_{d}-1).\qedhere
\end{align*}
\end{proof}
\begin{lem}
\label{lem:Gab-Sab}For $a,b\geq0$, we have 
\[
G_{a,b}(x-1)-S_{a,b}(x)=\frac{1}{x-1}G_{a-1,b}(x-1)-\frac{1}{x}S_{a,b-1}(x).
\]
\end{lem}
\begin{proof}
Applying Lemma \ref{lem:all_translation} to the case $d=a+b+1$ and
$(x_{1},\dots,x_{d})=(\overbrace{x,\dots,x}^{a},2x-1,\overbrace{x,\dots,x}^{b})$,
we have 
\begin{align*}
 & G_{a,b}(x-1)-S_{a,b}(x)\\
 & =f(x_{1}-1,\dots,x_{d}-1)-f(x_{1},\dots,x_{d})\\
 & =\frac{1}{x_{1}-1}f(x_{2}-1,\dots,x_{d}-1)-\frac{1}{x_{d}}f(x_{1},\dots,x_{d-1})\\
 & =(x-1)\tilde{f}(x_{1}-1,\dots,x_{d}-1,x-1)-x\tilde{f}(x,x_{1},\dots,x_{d})\\
 & =\frac{1}{x-1}G_{a-1,b}(x-1)-\frac{1}{x}S_{a,b-1}(x).\qedhere
\end{align*}
\end{proof}
\begin{lem}
\label{lem:one_translation}For $1\leq i\leq d$, we have
\begin{align*}
 & \tilde{f}(x_{0},\dots,x_{i-1},x_{i}-1,x_{i+1},\dots,x_{d+1})-\tilde{f}(x_{0},\dots,x_{d+1})\\
= & \frac{1}{x_{i}-x_{i-1}-1}\left(\tilde{f}(x_{0},\dots,x_{i-1},x_{i+1},\dots,x_{d+1})-\tilde{f}(x_{0},\dots,x_{i-2},x_{i}-1,x_{i+1},\dots,x_{d+1})\right)\\
 & +\frac{1}{x_{i}-x_{i+1}}\left(\tilde{f}(x_{0},\dots,x_{i},x_{i+2},\dots,x_{d+1})-\tilde{f}(x_{0},\dots,x_{i-1},x_{i+1},\dots,x_{d+1})\right).
\end{align*}
\end{lem}
\begin{proof}
Suppose that $1\leq i\leq d$. We define four subsets $S_{1},S_{2},S_{3},S_{4}$
of $\{(n_{0},\dots,n_{d+1})\in\mathbb{Z}^{d+2}\mid n_{0}=0,\; n_{d+1}=p\}$
by
\begin{align*}
S_{1}:= & \{0=n_{0}<\cdots<n_{i-1}<n_{i}-1<n_{i+1}<\cdots<n_{d+1}=p\},\\
S_{2}:= & \{0=n_{0}<\cdots<n_{i-1}=n_{i}-1<n_{i+1}<\cdots<n_{d+1}=p\},\\
S_{3}:= & \{0=n_{0}<\cdots<n_{i-1}<n_{i}<n_{i+1}<\cdots<n_{d+1}=p\},\\
S_{4}:= & \{0=n_{0}<\cdots<n_{i-1}<n_{i}=n_{i+1}<\cdots<n_{d+1}=p\}.
\end{align*}
Put 
\[
A_{k}:=\sum_{(n_{0},\dots,n_{d+1})\in S_{k}}\prod_{j=0}^{d+1}\frac{1}{n_{j}-x_{j}}
\]
for $k\in\{1,2,3,4\}$. It is immediate from the definition that
\begin{align*}
A_{1} & =\tilde{f}(x_{0},\dots,x_{i-1},x_{i}-1,x_{i+1},\dots,x_{d+1}),\\
A_{3} & =\tilde{f}(x_{0},\dots,x_{d+1}).
\end{align*}
We have
\begin{align*}
A_{2}= & \sum_{(n_{0},\dots,n_{d+1})\in S_{2}}\frac{1}{(n_{i-1}-x_{i-1})(n_{i}-x_{i})}\prod_{j\notin\{i-1,i\}}\frac{1}{n_{j}-x_{j}}\\
= & \frac{1}{1-x_{i}+x_{i-1}}\sum_{(n_{0},\dots,n_{d+1})\in S_{2}}\left(\frac{1}{n_{i-1}-x_{i-1}}-\frac{1}{n_{i}-x_{i}}\right)\prod_{j\notin\{i-1,i\}}\frac{1}{n_{j}-x_{j}}\\
= & \frac{1}{1-x_{i}+x_{i-1}}\left(\tilde{f}(x_{0},\dots,x_{i-1},x_{i+1},\dots,x_{d+1})-\tilde{f}(x_{0},\dots,x_{i-2},x_{i}-1,x_{i+1},\dots,x_{d+1})\right)
\end{align*}
and
\begin{align*}
A_{4} & =\sum_{(n_{0},\dots,n_{d+1})\in S_{4}}\frac{1}{(n_{i}-x_{i})(n_{i+1}-x_{i+1})}\prod_{j\notin\{i,i+1\}}\frac{1}{n_{j}-x_{j}}\\
 & =\frac{1}{x_{i}-x_{i+1}}\sum_{(n_{0},\dots,n_{d+1})\in S_{4}}\left(\frac{1}{n_{i}-x_{i}}-\frac{1}{n_{i+1}-x_{i+1}}\right)\prod_{j\notin\{i,i+1\}}\frac{1}{n_{j}-x_{j}}\\
 & =\frac{1}{x_{i}-x_{i+1}}\left(\tilde{f}(x_{0},\dots,x_{i},x_{i+2},\dots,x_{d+1})-\tilde{f}(x_{0},\dots,x_{i-1},x_{i+1},\dots,x_{d+1})\right).
\end{align*}
Since $S_{1}\sqcup S_{2}=S_{3}\sqcup S_{4}$, we have $A_{1}-A_{3}=-A_{2}+A_{4}$.
Thus
\begin{align*}
 & \tilde{f}(x_{0},\dots,x_{i-1},x_{i}-1,x_{i+1},\dots,x_{d+1})-\tilde{f}(x_{0},\dots,x_{d+1})\\
= & \frac{1}{x_{i}-x_{i-1}-1}\left(\tilde{f}(x_{0},\dots,x_{i-1},x_{i+1},\dots,x_{d+1})-\tilde{f}(x_{0},\dots,x_{i-2},x_{i}-1,x_{i+1},\dots,x_{d+1})\right)\\
 & +\frac{1}{x_{i}-x_{i+1}}\left(\tilde{f}(x_{0},\dots,x_{i},x_{i+2},\dots,x_{d+1})-\tilde{f}(x_{0},\dots,x_{i-1},x_{i+1},\dots,x_{d+1})\right).\qedhere
\end{align*}
\end{proof}
\begin{lem}
\label{lem:Sab-Gab}For $a,b\geq0$, we have
\[
S_{a,b}(x)-G_{a,b}(x)=\frac{1}{x(x-1)}f(\overbrace{x,\dots x}^{a+b})-\frac{1}{x-1}S_{a-1,b}(x)+\frac{1}{x}G_{a,b-1}(x)
\]
\end{lem}
\begin{proof}
By applying Lemma \ref{lem:one_translation} to the case $d=a+b+1$,
$i=a+1$ and $(x_{0},\dots,x_{d+1})=(\overbrace{x,\dots,x}^{a+1},2x,\overbrace{x,\dots,x}^{b+1})$,
we have
\begin{align*}
 & S_{a,b}(x)-G_{a,b}(x)\\
 & =x^{2}\left(\tilde{f}(\overbrace{x,\dots,x}^{a+1},2x-1,\overbrace{x,\dots,x}^{b+1})-\tilde{f}(\overbrace{x,\dots,x}^{a+1},2x,\overbrace{x,\dots,x}^{b+1})\right)\\
 & =\frac{x^{2}}{x-1}\left(\tilde{f}(\overbrace{x,\dots,x}^{a+b+2})-\tilde{f}(\overbrace{x,\dots,x}^{a},2x-1,\overbrace{x,\dots,x}^{b+1})\right)\\
 & \ \ \ +\frac{x^{2}}{x}\left(\tilde{f}(\overbrace{x,\dots,x}^{a+1},2x,\overbrace{x,\dots,x}^{b})-\tilde{f}(\overbrace{x,\dots,x}^{a+b+2})\right)\\
 & =\frac{1}{x(x-1)}f(\overbrace{x,\dots x}^{a+b})-\frac{1}{x-1}S_{a-1,b}(x)+\frac{1}{x}G_{a,b-1}(x)\qedhere
\end{align*}

\end{proof}

\subsection{Proof of the main theorem}

Put $U(x):=x^{p-1}/(x^{p-1}-1)$. Note that $f(\overbrace{x,\dots,x}^{d})=x^{-d}U(x)$
for $0<d<p$ by Lemma \ref{lem:Fxx}. By the argument in Section \ref{sub:Restatement},
Theorem \ref{thm:main_harmonic_sum} is equivalent to saying that
\begin{equation}
2G_{a,b}(x)=x^{-a-b-1}U(x)\label{eq:main_statement}
\end{equation}
for $a,b\geq0$ such that $a+b$ is an even integer less than $p-1$.
We will prove (\ref{eq:main_statement}) for $a,b\geq-1$ such that
$a+b$ is a non-negative even integer less than $p-1$ by induction
on $a+b$. First, (\ref{eq:main_statement}) is obvious if $a=-1$,
$b=-1$, or $(a,b)=(0,0)$.

Assume that $a\geq0$, $b\geq0$, and $a+b\geq2$. By Lemmas \ref{lem:Gab-Sab}
and \ref{lem:Sab-Gab}, we have
\begin{align}
G_{a,b}(x-1)-G_{a,b}(x)= & -\frac{1}{x}S_{a,b-1}(x)+\frac{1}{x}G_{a,b-1}(x)\nonumber \\
 & +\frac{1}{x-1}G_{a-1,b}(x-1)-\frac{1}{x-1}S_{a-1,b}(x)\nonumber \\
 & +\frac{U(x)}{x^{a+b+1}(x-1)}.\label{eq:diff}
\end{align}
By Lemma \ref{lem:Sab-Gab} and the induction hypothesis, for $b>0$
we have
\begin{align}
-\frac{1}{x}S_{a,b-1}(x)+\frac{1}{x}G_{a,b-1}(x)= & -\frac{f(\overbrace{x,\dots x}^{a+b-1})}{x^{2}(x-1)}+\frac{1}{x(x-1)}S_{a-1,b-1}(x)-\frac{1}{x^{2}}G_{a,b-2}(x)\nonumber \\
= & \frac{U(x)}{x^{a+b+1}}\left(-\frac{1}{x-1}-\frac{1}{2}\right)+\frac{1}{x(x-1)}S_{a-1,b-1}(x).\label{eq:part_1}
\end{align}
We can check by direct computation that (\ref{eq:part_1}) is also
true for $b=0$.

By Lemma \ref{lem:Gab-Sab} and the induction hypothesis, for $a>0$
we have
\begin{align}
\frac{1}{x-1}G_{a-1,b}(x-1)-\frac{1}{x-1}S_{a-1,b}(x)= & \frac{1}{(x-1)^{2}}G_{a-2,b}(x-1)-\frac{1}{x(x-1)}S_{a-1,b-1}(x)\nonumber \\
= & \frac{U(x-1)}{2(x-1)^{a+b+1}}-\frac{1}{x(x-1)}S_{a-1,b-1}(x).\label{eq:part_2}
\end{align}
We can check by direct computation that (\ref{eq:part_2}) is also
true for $a=0$.

By (\ref{eq:diff}), (\ref{eq:part_1}) and (\ref{eq:part_2}),
\[
G_{a,b}(x-1)-\frac{U(x-1)}{2(x-1)^{a+b+1}}=G_{a,b}(x)-\frac{U(x)}{2x^{a+b+1}}.
\]
Put $q(x):=G_{a,b}(x)-\frac{U(x)}{2x^{a+b+1}}$. Let 
\[
q(x)=\sum_{n=0}^{p-1}\sum_{k=0}^{\infty}\frac{c_{n,k}}{(x-n)^{k}}
\]
be a partial fractional decomposition. Since $q(x-1)=q(x)$, we have
$c_{0,k}=c_{1,k}=\cdots=c_{p-1,k}$. Since $a+b<p-1$, we see that
$q(x)$ has no pole at $x=0$, which implies that $c_{0,k}=0$. Thus
$q(x)=0$, which proves (\ref{eq:main_statement}).

Now we prove Theorem \ref{thm:main_finite_mzv}. The equality (\ref{eq:main_A})
is an immediate consequence of Theorem \ref{thm:main_harmonic_sum}.
To prove (\ref{eq:main_A_star}), we use the following well-known
lemma (for the proof of this lemma, see \cite[Prop 2.9]{Saito} for
example).
\begin{lem}
\label{lem:antipode}For $d\geq1$ and $k_{1},\dots,k_{d}\in\mathbb{Z}_{\geq1}$,
we have
\[
\sum_{j=0}^{d}(-1)^{j}\zeta_{\mathcal{A}}^{\star}(k_{1},\dots,k_{j})\zeta_{\mathcal{A}}(k_{d},\dots,k_{j+1})=0.
\]

\end{lem}
It is well known that
\begin{align}
\sum_{\substack{k_{1},\dots,k_{s}\geq1\\
k_{1}+\cdots+k_{s}=k
}
}\zeta_{\mathcal{A}}(k_{1},\dots,k_{s}) & =0,\label{eq:trivial_weight_zero}\\
\sum_{\substack{k_{1},\dots,k_{s}\geq1\\
k_{1}+\cdots+k_{s}=k
}
}\zeta_{\mathcal{A}}^{\star}(k_{1},\dots,k_{s}) & =0\label{eq:trivial_weight_zero_star}
\end{align}
for $s\geq1$ (see \cite{Hoffman}). By Lemma \ref{lem:antipode},
we have
\begin{equation}
\sum_{j=0}^{d}(-1)^{j}\sum_{\substack{k_{1},\dots,k_{d}\geq1\\
k_{1}+\cdots+k_{d}=k
}
}2^{k_{i}}\zeta_{\mathcal{A}}^{\star}(k_{1},\dots,k_{j})\zeta_{\mathcal{A}}(k_{d},\dots,k_{j+1})=0.\label{eq:total_wsum}
\end{equation}
By (\ref{eq:trivial_weight_zero}), we have
\begin{equation}
\sum_{\substack{k_{1},\dots,k_{d}\geq1\\
k_{1}+\cdots+k_{d}=k
}
}2^{k_{i}}\zeta_{\mathcal{A}}^{\star}(k_{1},\dots,k_{j})\zeta_{\mathcal{A}}(k_{d},\dots,k_{j+1})=0\ \ \ \ (i\leq j\leq d-1).\label{eq:part_wsum}
\end{equation}
By (\ref{eq:trivial_weight_zero_star}), we have
\begin{equation}
\sum_{\substack{k_{1},\dots,k_{d}\geq1\\
k_{1}+\cdots+k_{d}=k
}
}2^{k_{i}}\zeta_{\mathcal{A}}^{\star}(k_{1},\dots,k_{j})\zeta_{\mathcal{A}}(k_{d},\dots,k_{j+1})=0\ \ \ \ (1\leq j<i).\label{eq:part_wsum_star}
\end{equation}
By (\ref{eq:total_wsum}), (\ref{eq:part_wsum}), and (\ref{eq:part_wsum_star}),
we have
\[
\sum_{\substack{k_{1},\dots,k_{d}\geq1\\
k_{1}+\cdots+k_{d}=k
}
}2^{k_{i}}\zeta_{\mathcal{A}}(k_{d},\dots,k_{1})=(-1)^{d+1}\sum_{\substack{k_{1},\dots,k_{d}\geq1\\
k_{1}+\cdots+k_{d}=k
}
}2^{k_{i}}\zeta_{\mathcal{A}}^{\star}(k_{1},\dots,k_{d}).
\]
Therefore (\ref{eq:main_A_star}) follows from (\ref{eq:main_A}).
Thus Theorem \ref{thm:main_finite_mzv} is proved.

\section{\label{sec:Further-conjecture}Further conjectures}

We define a general weighted sum of finite multiple zeta values by
\[
W_{k}(n_{1},\dots,n_{d})=\sum_{\substack{k_{1},\dots,k_{d}\geq1\\
k_{1}+\cdots+k_{d}=k
}
}n_{1}^{k_{1}}\cdots n_{d}^{k_{d}}\zeta_{\mathcal{A}}(k_{1},\dots,k_{d}),
\]
\[
W_{k}^{\star}(n_{1},\dots,n_{d})=\sum_{\substack{k_{1},\dots,k_{d}\geq1\\
k_{1}+\cdots+k_{d}=k
}
}n_{1}^{k_{1}}\cdots n_{d}^{k_{d}}\zeta_{\mathcal{A}}^{\star}(k_{1},\dots,k_{d}).
\]
We list several conjectural weighted sum formulas for finite multiple
zeta values that are different from Theorem \ref{thm:main_finite_mzv}.
\begin{conjecture}
\label{conj:1123}For $k,r\in\mathbb{Z}_{\geq1}$, we have
\[
W_{k}(1,1,2,3,\dots,r-1,r)\stackrel{?}{=}0,
\]
\[
W_{k}^{\star}(1,1,2,3,\dots,r-1,r)\stackrel{?}{=}0.
\]

\end{conjecture}

\begin{conjecture}
\label{conj:11233}For $k\in\mathbb{Z}_{\geq1}$ and a positive odd
integer $r$, we have
\[
W_{k}(1,1,2,3,\dots,r-2,r-1,r,r)\stackrel{?}{=}0,
\]
\[
W_{k}^{\star}(1,1,2,3,\dots,r-2,r-1,r,r)\stackrel{?}{=}0.
\]

\end{conjecture}

\begin{conjecture}
\label{conj:abr}For $k,r\in\mathbb{Z}_{\geq1}$ and $a,b\in\mathbb{Q}$,
we have
\[
W_{k}(a,a+b,a+2b,\dots,a+rb)\stackrel{?}{=}W_{k}(b,a+b,a+2b,\dots,a+rb),
\]
\[
W_{k}^{\star}(a,a+b,a+2b,\dots,a+rb)\stackrel{?}{=}W_{k}^{\star}(b,a+b,a+2b,\dots,a+rb).
\]
\end{conjecture}
\begin{rem}
Conjecture \ref{conj:1123} is a special case of Conjecture \ref{conj:abr}.
\end{rem}

\section*{Acknowledgements}

This work was supported by JSPS KAKENHI Grant Numbers JP18J00982,
JP18K03243, and JP18K13392.

\bibliographystyle{plain}
\bibliography{WeightedSum}

\end{document}